\documentclass[12pt]{amsart}
\usepackage{amsmath, amscd}
\usepackage{amsfonts}
\usepackage{version}
\usepackage{amssymb}
\newcommand{\de}{\partial}
\newcommand{\R}{\mathbb R}

\newcommand{\al}{\alpha}

\newcommand{\C}{\mathbb C}

\newcommand{\B}{\mathbb B}

\newcommand{\Hol}{{\sf Hol}(\mathbb B^2, \mathbb C^2)}
\newcommand{\HD}{{\sf Hol}_D(\mathbb B^2, \mathbb C^2)}

\newcommand{\D}{\mathbb D}

\newcommand{\N}{\mathbb N}

\def\v{\varphi}

\def\Re{{\sf Re}\,}

\newtheorem{theorem}{Theorem}[section]
\newtheorem{lemma}[theorem]{Lemma}
\newtheorem{proposition}[theorem]{Proposition}
\newtheorem{corollary}[theorem]{Corollary}

\theoremstyle{definition}
\newtheorem{definition}[theorem]{Definition}

\theoremstyle{remark}
\newtheorem{remark}[theorem]{Remark}
\numberwithin{equation}{section}

\numberwithin{equation}{section}

\begin{document}
\title[Shearing process and support functions]{Shearing process and an example of a bounded support function in $S^0(\B^2)$}
\author[F. Bracci]{Filippo Bracci}\thanks{Supported by the ERC grant ``HEVO - Holomorphic Evolution Equations'' n. 277691}
\address{F. Bracci: Dipartimento di Matematica\\
Universit\`{a} di Roma \textquotedblleft Tor Vergata\textquotedblright\ \\
Via Della Ricerca Scientifica 1, 00133 \\
Roma, Italy} \email{fbracci@mat.uniroma2.it}

\begin{abstract}
We introduce a process, that we call {\sl shearing}, which for any given normal Loewner chain produces a normal Loewner chain made of shears automorphisms. As an application, and in stringent contrast to the one-dimensional case,  we prove the existence of a starlike bounded function in the class $S^0$ of the ball $\B^2$ (in fact the restriction of a shear automorphism of $\C^2$) which is a support point for a linear continuous functional.
\end{abstract}

\maketitle

\section{Introduction}

Let $\B^n:=\{z\in \C^n: \|z\|<1\}$. Let $S(\B^n)$ denote the family of univalent holomorphic maps from $\B^n$ to $\C^n$ such that $f(0)=0$ and $df_0={\sf id}$. Such a class is known to be compact for $n=1$ (see, {\sl e.g.} \cite{Po}) but it is not for $n>1$.

A {\sl normalized Loewner chain} on $\B^n$ is a family $(f_t)_{t\geq 0}$ of univalent maps from $\B^n$ to $\C^n$ such that $f_s(\B^n)\subset f_t(\B^n)$ for $0\leq s\leq t$, $f_t(0)=0$ and $d(f_t)_0=e^t {\sf id}$. A normalized Loewner chain $(f_t)$ is said a {\sl normal Loewner chain} if $\{e^{-t}f_t\}_{t\geq 0}$ is normal in $\B^n$. Let $S^0(\B^n)$ be the subset of $S(\B^n)$ of maps which admit {\sl parametric representation}, namely, $f\in S^0(\B^n)$ provided there exists a normal Loewner chain $(f_t)$ in $\B^n$ such that $f_0=f$. It is known, that in $\D=\B^1$ it holds $S^0(\D)=S(\D)$ (see \cite{Po}), while $S^0(\B^n)$ is strictly contained in $S(\B^n)$  for $n\geq 2$  (see, {\sl e.g.} \cite{GK03}). Nonetheless, the class $S^0(\B^n)$ is compact and many similar growth estimates as in the one-dimensional case can be pursued.

For studying extremal problems, it is important to determine the so called {\sl support points}. Endow ${\sf Hol}(\B^n, \C^n)$ with the  topology of uniform convergence on compacta, which makes it a Fr\'echet space. A function $f\in S^0(\B^n)$ is called a {\sl support point} if there exists a bounded linear functional $L:{\sf Hol}(\B^n, \C^n)\to \C$  such that $L$ is not constant on $S^0(\B^n)$ and $\Re L(f)=\max_{h\in S^0(\B^n)} \Re L(h)$.

In dimension one it was proved by Schaeffer (see, {\sl e.g.} \cite{Scha}) that support points are slit functions. In particular they are unbounded.

In higher dimension, much work has been done to study support points in $S^0(\B^n)$ and many evidences that support points in $S^0(\B^n)$ should be unbounded have been obtained, see, {\sl e.g.}, \cite{GHKK1}, \cite{GHKK2}, \cite{Schl}, \cite{GHK0}, \cite{BGHK}, \cite{M}, \cite{R}.

However, and very surprisingly, in this paper, we construct an example of a {\sl bounded} starlike map in $S^0(\B^2)$ which is a support point.

Such a map is quite simple and it is in fact the restriction to $\B^2$ of a shear of $\C^2$:
\[
\Phi(z_1,z_2)=(z_1+\frac{3\sqrt{3}}{2}z_2^2, z_2), \quad \forall (z_1,z_2)\in \B^2.
\]

Let $L^1_{0,2}: \Hol \to \C$ be defined as $L^1_{0,2} (f):=\frac{1}{2}\frac{\de^2 f_1}{\de z_2^2}(0)$, where $f=(f_1,f_2)\in \Hol$.

\begin{theorem}\label{example}
The map $\Phi\in S^0(\B^2)$  is starlike and  maximizes  $\Re L^1_{0,2}$.
\end{theorem}

The fact that $\Phi$ is starlike and belongs to $S^0(\B^2)$ was known from long (see  \cite[Example 3]{Su77}, \cite[Example 5]{RS}), however, we give a  proof of it in Lemma \ref{starlike}. The fact that $\Phi$ maximizes $\Re L^1_{0,2}$ follows at once from the following result:

\begin{theorem}\label{bound-intro}
Let $f(z_1,z_2)=(z_1+\sum_{\al\in \N^2, |\al|\geq 2} a^1_\al z^\al,z_2+\sum_{\al\in \N^2, |\al|\geq 2} a^2_\al z^\al)\in S^0(\B^2)$. Then $|a^1_{0,2}|\leq \frac{3\sqrt{3}}{2}$, and the bound is reached by $\Phi$.
\end{theorem}

The proof of Theorem \ref{bound-intro} relies on a process that we call {\sl shearing process} which has no  one-dimensional analogous and it seems to be interesting by itself. Let
\[
\HD:=\{h\in \Hol: h(0)=0 \hbox{ and $dh_0$ is diagonal and invertible}\}.
\]

\begin{definition}
Let $h\in \HD$ and write the Taylor expansion at $0$ as
\[
h(z)=(\lambda z_1+Az_2^2+O(|z_1|^2,|z_1z_2|,\|z\|^3), \mu z_2 + O(\|z\|^2)).
\]
Then we define the {\sl shearing of $h$} to be
\[
h^{[c]}(z_1,z_2):=(\lambda z_1+Az_2^2, \mu z_2)\quad \forall z\in \B^2.
\]
\end{definition}

The main result of the paper is the following

\begin{theorem}\label{shearingLoewner}
Let $(f_t)$ be a normal Loewner chain in $\B^2$. Then $(f_t^{[c]})$ is a normal Loewner chain in $\B^2$. In particular, if $f\in S^0(\B^2)$ then $f^{[c]}\in S^0(\B^2)$.
\end{theorem}

Such a theorem is proved in Section \ref{shearLow}. Using such a result, in Section \ref{bound} we prove Theorem \ref{bound-intro}.

We remark that our construction can be done  in any dimension greater than one, but for the sake of clearness, we give the results only in dimension $2$. Also, Theorem \ref{example} gives an example of a function in $S^0(\B^2)$ which is bounded by a certain $M>1$ but it is not a ``$\log M$-time reachable function'' in $S^0(\B^2)$ (see, {\sl e.g.}, \cite{GHK} for definition), contrarily to the one-dimensional case where the set of functions in $S^0(\D)$ which are bounded by $M>1$ coincides with ``$\log M$-time reachable functions'' in $S^0(\D)$.

\medskip

This work was done while the author was ``Giovanni Prodi chair'' in Nonlinear Analysis in Winter 2013/14 at Institut f\"{u}r Mathematik, Universit\"{a}t W\"{u}rzburg. He wishes to sincerely thank all the staff of the Chair of Complex Analysis for all the support and the  great atmosphere experienced there.

\section{The shearing process of Loewner chains}\label{shearLow}

\subsection{Shearing the class  $\mathcal M_{-}$}
Let
\[
\mathcal M_{-}:=\{H\in \Hol : H(0)=0, dH_0=-{\sf id}, \Re \langle H(z), z\rangle \leq 0 \quad \forall z\in \B^2\}.
\]
We show that the shearing of a holomorphic vector field in the class $\mathcal M_{-}$ is still in the class $\mathcal M_{-}$:

\begin{proposition}\label{cutoff}
If $H\in \mathcal M_{-}$, then  $H^{[c]}\in \mathcal M_{-}$.
\end{proposition}

\begin{proof}
It is clear that $H^{[c]}(0)=0, dH^{[c]}_0={\sf -id}$, so we have to check that $\Re \langle H^{[c]}(z),z\rangle \leq 0$ for all $z\in \B^2$. Let  write the Taylor expansion of $H$ at $0$ as
\begin{equation}\label{expH}
H(z)=(-z_1+\sum_{\al\in \N^2: |\al|\geq 2} q^1_\al z^\al, -z_2+\sum_{\al\in \N^2: |\al|\geq 2} q^2_\al z^\al).
\end{equation}
We know that for all $z\in \B^2$
\begin{equation}\label{first}
0\geq  \Re \langle H(z), z\rangle =-|z_1|^2-|z_2|^2+\sum_{\al \in \N^2, |\al|\geq 2} \Re \left(q_\al^1z^\al\overline{z_1}\right)+\sum_{\al \in \N^2, |\al|\geq 2} \Re \left(q_\al^2z^\al\overline{z_2}\right).
\end{equation}
Let $\eta\in [0,2\pi)$ be such that $q^1_{0,2}e^{-i\eta}=|q^1_{0,2}|$. Take $z_1=xe^{i(\theta+\eta)}$ and $z_2=ye^{i\frac{\theta}{2}}$ for $x,y\geq 0, x^2+y^2<1$ and $\theta \in \R$. Substituting these expressions in \eqref{first} we obtain
\begin{equation}\label{second}
\begin{split}
0&\geq  -x^2-y^2+|q^1_{0,2}|xy^2+\sum_{\al \in \N^2, |\al|\geq 2, \al\neq (0,2)} x^{\al_1+1}y^{\al_2}\Re \left(q_\al^1 e^{i(\al_1-1)\eta}e^{i[\al_1+\frac{\al_2}{2}-1]\theta}\right)\\&+\sum_{\al \in \N^2, |\al|\geq 2} x^{\al_1}y^{\al_2+1}\Re \left(q_\al^2e^{i\al_1\eta}e^{i[\al_1+\frac{\al_2}{2}-\frac{1}{2}]\theta}\right).
\end{split}
\end{equation}
Now, notice that among all $\al\in \N^2$ with $|\al|\geq 2$ the expression $\al_1+\frac{\al_2}{2}-1=0$ has only the solution $\al=(0,2)$, while $\al_1+\frac{\al_2}{2}-\frac{1}{2}=0$ has no solution for $\al\in \N^2$ with $|\al|\geq 2$. Therefore, all terms in the previous expression except the first three terms are of the form $a \cos (m\theta)$, $a\cos (m\frac{\theta}{2})$ (or with $\sin$ instead of $\cos$), for some $a\in \R$ and $m\in \N, m\geq 1$. Thus, if we integrate \eqref{second} in $\theta$ for $\theta\in [0,4\pi]$ and taking into account that the series converges uniformly on compacta and thus we can exchange the series with the integral, all the terms in \eqref{second} but the first three, become zero and we obtain:
\begin{equation}\label{stimmate}
0\geq  -x^2-y^2+|q^1_{0,2}|xy^2 \quad \forall x,y\geq 0, x^2+y^2<1.
\end{equation}
From this it follows that
\begin{equation*}
\Re \langle H^{[c]}(z), z \rangle = -|z_1|^2-|z_2|^2-\Re \left(q^1_{0,2} z_2^2\overline{z_1}\right)\leq -|z_1|^2-|z_2|^2+|q^1_{0,2}||z_1||z_2|^2\leq 0,
\end{equation*}
proving that $H^{[c]}\in \mathcal M_{-}$.
\end{proof}

\begin{corollary}\label{stimaM}
Let $H\in \mathcal M_{-}$ be given by \eqref{expH}. Then $|q^1_{0,2}|\leq \frac{3\sqrt{3}}{2}$.
\end{corollary}

\begin{proof}
By Proposition \ref{stimaM} it is enough to prove  the result for  $H(z)=(-z_1+az_2^2, -z_2)$. By \eqref{stimmate}, we have
\[
 -x^2-y^2+|a|xy^2\leq 0 \quad \forall x,y\geq 0, x^2+y^2<1.
\]
Studying such a function it is not hard to show that the equation holds  only if $|a|\leq \frac{3\sqrt{3}}{2}$.
\end{proof}

The previous estimate is sharp, indeed we have the following result:

\begin{lemma}\label{starlike}
Let $H(z)=(-z_1+\frac{3\sqrt{3}}{2} z_2^2, -z_2)$. Then $H\in \mathcal M_{-}$. Moreover,
\[
f_t(z)=e^t\left(z_1+\frac{3\sqrt{3}}{2}z_2^2, z_2\right)
\]
 is a normal Loewner chain which satisfies the Loewner PDE
\[
\frac{\de f_t(z)}{\de t}=d(f_t)_z H(z).
\]
In particular, $\Phi=f_0\in S^0(\B^2)$ and $\Phi$ is starlike.
\end{lemma}
\begin{proof}
A simple computation shows that $H\in \mathcal M_{-}$ and that $(f_t)$ satisfies the Loewner PDE and the hypothesis of \cite[Thm. 8.1.6]{GK03}, so it is a normal Loewner chain. Alternatively, one can solve the Loewner ODE $\frac{\de \v_{s,t}(z)}{\de t}=H(\v_{s,t}(z))$, $\v_{s,s}(z)=z$ (which is in fact a semigroup equation) and check that $f_s=\lim_{t\to \infty}e^{t}\v_{s,t}$ for $s\in \R^+$. Hence, by \cite[Thm 8.1.6]{GK03} it follows that   $(f_t)$ is a normal Loewner chain. It is finally well known (see, {\sl e.g.} \cite[Thm. 8.2.1]{GK03}) that $f$ is starlike if and only if $(e^t f)$ is a normal Loewner chain.
\end{proof}

\subsection{Shearing normal Loewner chains} We show that the shearing of a normal Loewner chain is still a normal Loewner chain.

The first simple observation  is the following lemma, which can be checked by Taylor expansion at $0$:
\begin{lemma}\label{compo-l}
Let $h,g\in  \HD$ and suppose $h\circ g$ is well defined. Then for all $z\in \B^2\cap g^{[c]}(\B^2)$
\begin{equation}\label{compo}
(h\circ g)^{[c]}(z)=h^{[c]}(g^{[c]}(z)).
\end{equation}
\end{lemma}
In fact, \eqref{compo} holds for all $z\in \C^2$ if one denotes  by $h^{[c]}$  the corresponding shear and not just its restriction to $\B^2$.

Now we can prove Theorem \ref{shearingLoewner}:
\begin{proof}[Proof of Theorem \ref{shearingLoewner}] A {\sl Herglotz vector field $G(z,t)$ associated with the class $\mathcal M_{-}$ } is a map $G:\R^+\times \B^2\to \C^2$ such that
\begin{itemize}
\item[(i)] The mapping $G(z,\cdot)$ is measurable on $\R^+$ for all
$z\in \B^2$.
\item[(ii)] $G(\cdot, t)\in \mathcal M_{-}$  for a.e. $t\in [0,+\infty)$.
\end{itemize}

Let $(f_t)$ be a normal Loewner chain and set $\v_{s,t}:=f_t^{-1}\circ f_s$ for $0\leq s\leq t$. It is known (see, {\sl e.g.} \cite[Chapter 8]{GK03}) that $(f_t)$ is absolutely continuous in $t$ locally uniformly in $z$ and that there exists a Herglotz vector field $G(z,t)$ associated with the class $\mathcal M_{-}$ such that the following Loewner ODE is satisfied:
\[
\frac{\de \v_{s,t}(z)}{\de t}=G(\v_{s,t}(z),t) \quad a.e.\ t\geq 0, \forall z\in \B^2.
\]
By Proposition \ref{cutoff}, $G^{[c]}(z,t)$ is a Herglotz vector field associated with the class $\mathcal M_{-}$. By Lemma \ref{compo-l}, it follows that
\[
\frac{\de \v^{[c]}_{s,t}(z)}{\de t}=G^{[c]}(\v^{[c]}_{s,t}(z),t) \quad a.e.\ t\geq s, \forall z\in \v_{s,t}^{[c]}(\B^2)\cap \B^2.
\]
In fact, the previous equation holds for all $z\in \C^2$ if  considered at level of shears in $\C^2$.

Thus, since by \cite[Theorem 8.1.3]{GK03} the solution  to the Loewner ODE is unique, in particular, for all $0\leq s\leq t$, $\v^{[c]}_{s,t}(\B^2)\subset \B^2$. Hence, since again by Lemma \ref{compo-l} we have $f_s^{[c]}=f_t^{[c]}\circ \v_{s,t}^{[c]}$ for all $0\leq s\leq t$, it follows that $f^{[c]}_s(\B^2)\subset f_t^{[c]}(\B^2)$ for all $0\leq s\leq t$. It is now easy to check that $(f_t^{[c]})$ is a normal Loewner chain.
\end{proof}

\section{Sharp bound for the coefficient $a_{0,2}^1$}\label{bound}

\begin{theorem}\label{bound-in}
Let $f=(z_1+\sum_{\al\in \N^2: |\al|\geq 2} a_\al^1 z^\al, z_2+\sum_{\al\in \N^2: |\al|\geq 2} a_\al^2 z^\al)\in S^0(\B^2)$. Then $|a^1_{0,2}|\leq \frac{3\sqrt{3}}{2}$. Such an estimate is sharp and it is reached by $\Phi$.
\end{theorem}

\begin{proof}
Let $(f_t)$ be a normal Loewner chain such that $f_0=f$. By Theorem \ref{shearingLoewner}, $f^{[c]}\in S^0(\B^2)$ and it has a parametric representation given by $f_t(z)=(e^t z_1+a(t) z_2^2, e^t z_2)$, where $a:\R^+\to \R$ is a bounded absolutely continuous function and $a(0)=a^1_{0,2}$. Let $\v_{s,t}:=f_t^{-1}\circ f_s$. As in the proof of Theorem \ref{shearingLoewner}, there exists a Herglotz vector field $G(z,t)$ associated with the class $\mathcal M_{-}$ such that
\begin{equation}\label{prime}
\frac{\de \v^{[c]}_{s,t}(z)}{\de t}=G^{[c]}(\v^{[c]}_{s,t}(z),t) \quad a.e.\ t\geq s, \forall z\in \B^2.
\end{equation}
Write $G^{[c]}(z,t)=(-z_1+q(t)z_2^2, -z_2)$ and $\v^{[c]}_{s,t}=(e^{s-t}z_1+a(s,t)z_2^2, e^{s-t}z_2)$ for $0\leq s\leq t$.  Writing down explicitly \eqref{prime}, we obtain
\begin{equation*}
\begin{cases}
\frac{\de a(s,t)}{\de t}&=-a(s,t)+q(t)e^{2(s-t)} \quad a.e.\ t\geq s, \\
a(s,s)&=0
\end{cases}
\end{equation*}
from which it follows that for $0\leq s\leq t$
\[
a(s,t)=e^{s-t}\int_s^t q(\tau)e^{s-\tau}d\tau.
\]
By Corollary \ref{stimaM} we get $|a(s,t)|\leq  \frac{3\sqrt{3}}{2}e^{s-t}(1-e^{s-t})$. By \cite[Thm. 8.1.5]{GK03} $\lim_{t\to \infty} e^{t}\v^{[c]}_{s,t}=f^{[c]}_s$ (uniformly on compacta) for all $s\in \R^+$. Hence
\[
|a(s)|=\lim_{t\to \infty}|e^t a(s,t)|\leq \lim_{t\to \infty} e^t\frac{3\sqrt{3}}{2}e^{s-t}(1-e^{s-t})=\frac{3\sqrt{3}}{2}e^s,
\]
from which $|a_{0,2}^1|\leq \frac{3\sqrt{3}}{2}$.
\end{proof}

\begin{remark}
Let $\Phi_a(z_1,z_2):=(z_1+az_2^2, z_2)$ for $a\in \C$. It was known that $\Phi_a\in S^0(\B^2)$ for $|a|\leq \frac{3\sqrt{3}}{2}$ since such a map is starlike (see \cite[Example 3]{Su77}), while it was known that $\Phi_a\not\in S^0(\B^2)$ for $|a|\geq 2\sqrt{15}$ (see \cite[Remark 3.5]{GHKK}) since it does not satisfy the growth estimates for the class $S^0(\B^2)$. In fact, due to Theorem \ref{bound-in}, $\Phi_a\not\in S^0(\B^2)$ for $|a|> \frac{3\sqrt{3}}{2}$.
\end{remark}

\bibliographystyle{amsplain}

\end{document}